\documentclass{article}
\title{An improved explicit bound on $|\zeta(\frac{1}{2} + it)|$}
\author{D.J. Platt\\ Heilbronn Institute for Mathematical Research \\ University of Bristol, Bristol, UK\\ dave.platt@bris.ac.uk\\ \\
and\\ \\
T.S. Trudgian\footnote{Supported by Australian Research Council DECRA Grant DE120100173.}\\
Mathematical Sciences Institute\\ The Australian National University,
 ACT 0200, Australia\\ timothy.trudgian@anu.edu.au
}
\usepackage{url}
\usepackage{amsthm}
\usepackage{amsmath}
\usepackage{comment}
\usepackage{amssymb}
\usepackage{booktabs}
\newtheorem{thm}{Theorem}
\newtheorem{Lem}{Lemma}
\newtheorem{cor}{Corollary}

\begin{document}

\maketitle
\begin{abstract}
\noindent
This article proves the bound $|\zeta(\frac{1}{2} + it)|\leq 0.732 t^{\frac{1}{6}} \log t$ for $t \geq 2$, which improves on a result by Cheng and Graham. We also show that $|\zeta(\frac{1}{2}+it)|\leq 0.732 |4.678+it|^{\frac{1}{6}} \log |4.678+it|$ for all $t$.
\end{abstract}

\section{Introduction}
The Riemann zeta-function $\zeta(s)$ is known \cite{Huxley} to satisfy $\zeta(\frac{1}{2} + it)\ll_{\epsilon} t^{\frac{32}{205} + \epsilon}$ for all $t\gg 1$ and for every $\epsilon>0$. Explicit estimates of the sort
\begin{equation*}\label{exp}
|\zeta(\tfrac{1}{2} + it)| \leq k_{1} t^{k_{2}} (\log t)^{k_{3}}, \quad (t\geq t_{0})
\end{equation*}
are difficult to produce since, attempts at small values of $k_{2}$ lead to complicated arguments in the calculation of $k_{1}$.
Using the approximate functional equation and the Riemann--Siegel formula one may show that
\begin{equation}\label{lehm}
|\zeta(\tfrac{1}{2} + it)|\leq \frac{4}{(2\pi)^{\frac{1}{4}}} t^{\frac{1}{4}}, \quad (t\geq 0.2).
\end{equation}
Lehman \cite[Lem.\ 2]{Lehman} proved this for $t\geq 128\pi$ --- see also \cite[Thm 2]{Titchmarsh3} and \cite[Thm 1]{Turing} --- one may verify that (\ref{lehm}) holds in the range $0.2\leq t < 128\pi$ by direct computation.
The only other result of which we are aware is due to Cheng and Graham \cite{Cheng}, viz.
\begin{equation}\label{CG}
|\zeta(\tfrac{1}{2} + it)|\leq 3 t^{\frac{1}{6}}\log t, \quad (t\geq e).
\end{equation}
The upper bound in (\ref{CG}) is smaller than that in (\ref{lehm}) when $t\geq 1.4\times 10^{21}$. This is unfortunate since for some problems one seeks information for $t\geq T_{0}$, where $T_{0}$ is at most the height to which the Riemann hypothesis has been verified. The first author \cite{Plattarxiv} has confirmed that for $0\leq t \leq 3.06 \times 10^{10}$ all non-trivial zeroes of $\zeta(\sigma + it)$ lie on the critical line.

In \cite[(5.4)]{TrudgianS2} the second author showed that one could combine Theorem 3 of \cite{Cheng} with (\ref{lehm}) to show that
\begin{equation*}\label{CGT}
|\zeta(\tfrac{1}{2} + it)|\leq 2.38 t^{\frac{1}{6}}\log t, \quad (t\geq e),
\end{equation*}
which is better than the bound in (\ref{lehm}) only when $t\geq 10^{19}$.
The purpose of this article is to revisit the paper by Cheng and Graham and to prove

\begin{thm}\label{t1}
$$|\zeta(\tfrac{1}{2} + it)|\leq 0.732 t^{\frac{1}{6}} \log t, \quad (t\geq 2).$$
\end{thm}
The bound in Theorem \ref{t1} improves on that in (\ref{lehm}) whenever $t\geq 5.868\times 10^{9}$. Three applications are apparent: \cite{RamDraft, TrudgianTuring,TrudgianS2} which respectively relate to explicit estimates for zero-density theorems, bounding $\int_{0}^{T}S(t)\, dt$, and bounding $S(t)$, where $\pi S(t)$ is the argument of the zeta-function on the critical line. 
More precisely, when $t$ does not coincide with an ordinate of a zero of $\zeta(\sigma +it)$, $S(t)$ is defined as
\begin{equation*}
S(t) = \pi^{-1}\arg\zeta(\tfrac{1}{2} +it),
\end{equation*}
where the argument is determined via continuous variation along the straight lines connecting $2, 2+it$ and $\frac{1}{2}+it$, with $S(0) = 0$. If $t$ is such that $\zeta(\sigma + it) = 0$ then define $S(t)$  to be $\frac{1}{2}\lim_{\epsilon \rightarrow 0}\{S(t-\epsilon) + S(t+\epsilon)\}.$

The estimate for $S(t)$ can be improved immediately to give
\begin{cor}\label{c1}
If $T\geq e$, then
$$|S(T)| \leq 0.110\log T + 0.290\log\log T + 2.290.$$
\end{cor}
\begin{proof}
Using Theorem \ref{t1} one may take $(k_{1}, k_{2}, k_{3}) = (0.732, 1/6, 1)$ in \cite[(4.8)]{TrudgianS2}. Instead of choosing $Q_{0} = 2$ on page 291 of \cite{TrudgianS2}, we choose $Q_{0} = 5.$ The choice of $\eta = 0.064, r = 2.032$ on the same page establishes Corollary \ref{c1}.
\end{proof}
This improves the constant term in Theorem 1 \cite{TrudgianS2} from 2.510 to 2.290.

The improvement of Theorem \ref{t1} over the result in \cite{Cheng} comes from two ideas. First, an explicit form of the `standard' approximate functional equation is used (cf.\ Lemma \ref{HL}), in which one needs to estimate sums of the form $\sum_{n\leq Y} n^{it}$, where $t^{\frac{1}{2}} \ll Y \ll t^{\frac{1}{2}}$. This requires only one round of applying estimates for exponential sums. Cheng and Graham considered an approximation to $\zeta(\frac{1}{2} + it)$ in which one needs to estimate a longer sum with $t \ll Y \ll t$. They require two different estimates for exponential sums to cover this range. 
Second, some minor adjustments are made to some of the results in \cite{Cheng}, and more variables are optimised.

We prove some necessary lemmas in \S\ref{sec2}. We prove Theorem \ref{t1} for large $t$ in \S\ref{sec large} and for small $t$ in \S\ref{sec small}. We conclude with some computational remarks in \S\ref{secconc}.

\subsection*{Acknowledgements} 
We are grateful to Olivier Ramar\'{e} for helpful suggestions and comments.

\section{Preparatory Lemmas}\label{sec2}
It is necessary to record some estimates for exponential sums. Versions of the following lemmas without explicit constants can be found in \cite[Thm 5.9 and Lemma 5.10]{Titchmarsh}. Slightly coarser explicit versions can be found in \cite[p.\ 36]{KKapprox} and~\cite[Lemma~2.2]{Hab}
\begin{Lem}\label{CLem}
Assume that $f(x)$ is a real-valued function with two continuous derivatives when $x\in[N+1, N+L]$. If there exist two real numbers $V<W$ with $W>1$ such that
\begin{equation*}
\frac{1}{W}\leq |f''(x)|\leq \frac{1}{V}
\end{equation*}
for $x\in[N+1, N+L]$, then
\begin{equation*}
\bigg|\sum_{n=N+1}^{N+L} e^{2\pi i f(n)}\bigg| \leq \left(\frac{L-1}{V} + 1\right)\left(2\sqrt{\frac{2}{\pi}}W^{1/2} + 2\right) +1.
\end{equation*}
\end{Lem}
\begin{proof}
This is Lemma 3 in \cite{Cheng} with three slight adjustments. First, when applying the mean-value theorem on the first line of page 1268 of \cite{Cheng} one obtains $k \leq (L-1)/V + 2$ instead of $k\leq L/V + 2$. Second, when estimating the $2(k-1)$ intervals trivially, one may note that there are two intervals of length $W \Delta + 1$, namely those intervals from $(C_{k}- \Delta, C_{k})$ and $(C_{1}, C_{1} + \Delta)$, whereas there are $k-2$ intervals of length $2W \Delta + 1$. Third, we retain the constant $2 \sqrt{2/\pi}$ as opposed to (the only slightly larger) $8/5$.
\end{proof}

\begin{Lem}\label{L2}
Let $f(n)$ be a real-valued function and let $M$ be a positive integer. Then
\begin{equation}\label{5.10}
\bigg| \sum_{n=N+1}^{N+L} e^{2\pi i f(n)}\bigg|^{2} \leq \frac{L(L+M-1)}{M} + \frac{2(L+M-1)}{M} \sum_{m=1}^{M-1} \left( 1 - \frac{m}{M}\right) \max_{K\leq L} \bigg|\sum_{m, K}\bigg|,
\end{equation}
where
\begin{equation*}
\sum_{m, K} = \sum_{n=N+1}^{N+K} e^{2\pi i ( f(n+m) - f(n))}.
\end{equation*}
\end{Lem}
\begin{proof}
This is Lemma 5 in \cite{Cheng} with $L+M$ changed to $L+M-1$, a substitution that is clearly permitted as per the displayed equation at the bottom of \cite[p.~1272]{Cheng}. This differs from Lemma 5.10 in \cite{Titchmarsh} in three respects: there is no upper restriction on $M$, the coefficients are smaller (in \cite{Titchmarsh} both terms in (\ref{5.10}) have 4 as their leading coefficients), and the factor $(1-m/M)$ is present.
\end{proof}

\begin{Lem}\label{HL}
For $t\geq 100$,
\begin{equation}\label{lem1eq}
|\zeta(\tfrac{1}{2} +it)| \leq 2 |\sum_{n\leq \sqrt{\frac{t}{2\pi}}} n^{-\frac{1}{2} - it}| + 1.53 t_{0}^{-\frac{1}{4}} + 3.23 t_{0}^{-\frac{3}{4}}.
\end{equation}
\end{Lem}
\begin{proof}
We use Theorem 1 \cite{Titchmarsh3}, from which it follows that
\begin{equation}\label{cook}
|\zeta(\tfrac{1}{2} + it)| \leq 2 |\sum_{n\leq \sqrt{\frac{t}{2\pi}}} n^{-\frac{1}{2} - it}| + \frac{|\Gamma(\frac{1}{2} + it)|}{2\pi} e^{\frac{1}{2} \pi t} (2\pi)^{\frac{1}{2}} |g(\tfrac{t}{2\pi})| + |R(s)|,
\end{equation}
where, in Titchmarsh's expression for $R(s)$, there appears to be a blemish on the page: the $^{8}$ ought to be $\frac{8}{3}$, as per equation (4.1) of \cite{Titchmarsh3}. By the last line on p.\ 235 of \cite{Titchmarsh3}) we have
\begin{equation}\label{gdef}
|g(\tfrac{t}{2\pi})| \leq (2\pi)^{\frac{1}{4}} t^{-\frac{1}{4}} \bigg| \frac{\cos 2\pi(x^{2} -x - \frac{1}{16})}{\cos 2\pi x}\bigg|,
\end{equation}
where $0\leq x \leq 1$. By (2.6) and (2.7) of \cite{LehmanOld} we have $|g(\frac{t}{2\pi})| \leq (\cos \frac{\pi}{8}) (2\pi)^{\frac{1}{4}} t^{-\frac{1}{4}}$. With the version of Stirling's theorem given in Lemma $\epsilon$ in \cite{Titchmarsh3} we can now bound the second term in (\ref{cook}). Finally, using Titchmarsh's expression for $R(s)$, we note that $R(s)t^{-\frac{3}{4}}$ is decreasing in $t$ provided that $t> (5/2)^{3}$. A computation of the constants involved proves the lemma.
\end{proof}

\section{Proof of Theorem \ref{t1} for large $t$.}\label{sec large}
Write the sum in (\ref{lem1eq}) as $$\sum_{n \leq A_{0} t^{\frac{1}{3}}}n^{-\frac{1}{2} -it} + \sum_{A_{0} t^{\frac{1}{3}} < n \leq \sqrt{\frac{t}{2\pi}}}n^{-\frac{1}{2} -it}$$
provided that the interval of summation in the second sum is non-empty, that is, provided that
\begin{equation}\label{con1}
t_{0} > A_{0}^{6} (2\pi)^{3}.
\end{equation}
The trivial estimate gives
\begin{equation*}\label{par1}
\bigg|\sum_{n\leq A_{0} t^{\frac{1}{3}}} n^{-\frac{1}{2} - it}\bigg| \leq \sum_{n\leq A_{0} t^{\frac{1}{3}}} n^{-\frac{1}{2}} \leq 2 A_{0}^{\frac{1}{2}} t^{\frac{1}{6}} - 1.
\end{equation*}
Now consider 
\begin{equation*}\label{xdef}
X_{j} = A_{0} k^{j} t^{\frac{1}{3}}, 
\end{equation*}
where $k>1$ is a parameter to be determined later, and $j= 0, 1, 2, \ldots, J$, where
\begin{equation*}\label{Jdef}
J \leq \frac{\frac{1}{6} \log t - \log \left(A_{0}(2\pi)^{\frac{1}{2}}\right)}{\log k} + 1.
\end{equation*}
Also, let $N_{j} = [X_{j}]$ be the integer part of $X_{j}$.
It follows that
$$
\sum_{A_{0} t^{\frac{1}{3}} < n \leq \sqrt{\frac{t}{2\pi}}} n^{-\frac{1}{2} - it} = \sum_{j=1}^{J} \sum_{n = N_{j-1} +1}^{\min\{N_{j}, \sqrt{\frac{t}{2\pi}}\}} n^{-\frac{1}{2} -it},
$$
whence, by partial summation we have
\begin{equation}\label{par}
\bigg|\sum_{A_{0} t^{\frac{1}{3}} < n \leq \sqrt{\frac{t}{2\pi}}} n^{-\frac{1}{2} - it}\bigg| \leq \sum_{j=1}^{J} \frac{1}{X_{j-1}^{\frac{1}{2}}} \max_{L \leq N_{j} - N_{j-1}} \bigg| \sum_{n= N_{j-1} +1}^{N_{j-1} + L} e^{-it\log n}\bigg|.
\end{equation}
Denote the sum over $n$ in (\ref{par}) by $S_{j}$. We may estimate $S_{j}$ using Lemmas \ref{CLem} and \ref{L2}. First apply Lemma \ref{L2} to $S_{j}$ and thence apply Lemma \ref{CLem} to the resulting 
\begin{equation*}
\sum_{m, K} = \sum_{n= N_{j-1} +1}^{N_{j-1} + K} e^{-it \left(\log (n+m) - \log n\right)}.
\end{equation*} Choose $M = [k^{j} \theta] + 1$, for some $\theta$ to be determined later. We impose the addition restriction that $M\geq 2$ so as to use the bounds in (\ref{edited}) without worry.

We need to determine $V$ and $W$ in Lemma \ref{CLem}. We have
$$f(x) = -\frac{t}{2\pi}\left( \log (x+m) - \log x\right), \quad |f''(x)| = \frac{tm}{2\pi} \left(\frac{m + 2x}{x^{2}(x+m)^{2}}\right).$$ 
Since $(m+2x)/(x(x+m))^{2}$ is decreasing in both $x$ and $m$ we take $m=0, x=A_{0}k^{j-1} t^{\frac{1}{3}}$, and $m= M-1 \leq k^{j} \theta, x = A_{0} k^{j} t^{\frac{1}{3}}$ to find that $1/W \leq |f''(x)| \leq 1/V$, where
\begin{equation*}\label{VW}
V = \frac{\pi A_{0}^{3} k^{3j}}{k^{3} m}, \quad W = \frac{\pi k^{3j} A_{0}^{3}}{m} \left( 1 + \frac{\theta}{A_{0}t_{0}^{\frac{1}{3}}}\right)^{2}.
\end{equation*}
In order to apply Lemma \ref{CLem} it remains only to note that
\begin{equation}\label{Ldef}
L \leq (k-1) X_{j-1} + 1 \leq (k-1) A_{0} k^{j-1} t^{\frac{1}{3}} +1.
\end{equation}
One may now apply Lemma \ref{CLem} to find that
\begin{equation*}\label{postl2}
|\sum_{m, K}| \leq A_{1}t^{\frac{1}{3}} m^{\frac{1}{2}} k^{-\frac{1}{2} j} + A_{2} t^{\frac{1}{3}} m k^{-2j} + A_{3} m^{-\frac{1}{2}} k^{\frac{3}{2}j} + 3,
\end{equation*}
where
\begin{equation*}\label{Adef}
A_{1} = \frac{2\sqrt{2}(k-1)k^{2}Y_{0}}{\pi A_{0}^{\frac{1}{2}}}, \quad A_{2} = \frac{2(k-1)k^{2}}{\pi A_{0}^{2}},
\quad A_{3} =  2\sqrt{2} A_{0}^{\frac{3}{2}}Y_{0},
\quad Y_{0} =1 + \frac{\theta}{A_{0} t_{0}^{\frac{1}{3}}}.
\end{equation*}
 One of the advantages of using Lemma \ref{CLem} over Lemma 3 in \cite{Cheng} is that, according to (\ref{Ldef}), $L-1$ generates only one term.

The displayed formulae on page 1277 of \cite{Cheng} show that
\begin{equation}\label{edited}
\sum_{1\leq m \leq M -1} \left( 1 - \frac{m}{M}\right)m^{\frac{1}{2}} \leq \frac{4}{15} M^{\frac{3}{2}}, \quad \sum_{1\leq m \leq M -1} \left( 1 - \frac{m}{M}\right)m^{-\frac{1}{2}} \leq \frac{4}{3} M^{\frac{1}{2}}.
\end{equation}
Applying this gives
\begin{equation*}
\frac{1}{M} \sum_{m=1}^{M-1}\left(1 - \frac{m}{M}\right)|\sum_{m, K}| \leq \frac{4}{15}A_{1}t^{\frac{1}{3}} M^{\frac{1}{2}} k^{-\frac{1}{2}j} + \frac{1}{6}A_{2} t^{\frac{1}{3}} M k^{-2j} + \frac{4}{3}A_{3} M^{-\frac{1}{2}} k^{\frac{3}{2}j} + \frac{3}{2}.
\end{equation*}
Return now to Lemma \ref{L2}
\begin{equation*}
|S_{j}|^{2} \leq \frac{L(L+M-1)}{M} + 2(L+M-1)\left(\frac{1}{M} \sum_{m=1}^{M-1}\left(1 - \frac{m}{M}\right)|\sum_{m, K}| \right).
\end{equation*}
For $\alpha>0$, $(L+M-1)M^{\alpha}$ is an increasing function of $M$; $(L+M-1)/M$ is decreasing. We use an upper bound for the numerator and a lower bound for the denominator in $(L+M- 1)/M^{1/2}$. With $M= [k^{j} \theta] +1$ we have, 
\begin{equation*}\label{enterb}
|S_{j}|^{2} \leq B_{1} k^{j} t^{\frac{2}{3}} + B_{2} t^{\frac{2}{3}} + B_{3} k^{j} t^{\frac{1}{3}} + B_{4} k^{2j} t^{\frac{1}{3}},
\end{equation*}
where
\begin{equation}\label{red}
\begin{split}
A_{4} &= \frac{(k-1)^{2} A_{0}^{2}}{k^{2}\theta}\left( 1 + \frac{1}{(k-1) A_{0}  t_{0}^{\frac{1}{3}}}\right) \left( 1  + \frac{\theta k}{(k-1)A_{0}t_{0}^{\frac{1}{3}}}\right)\\
A_{5} &= \frac{2(k-1)A_{0}}{k} \left( 1 + \frac{1}{(k-1)A_{0}t_{0}^{\frac{1}{3}}} + \frac{\theta k}{(k-1) A_{0} t_{0}^{\frac{1}{3}}}\right)\\
A_{6} &= \frac{4}{15} A_{1} \theta^{\frac{1}{2}} \left( 1 + \frac{1}{k\theta}\right)^{\frac{1}{2}}, \quad A_{7}= \frac{A_{2}\theta}{6} \left( 1 + \frac{1}{k\theta}\right)\\
A_{8} &= \frac{4 A_{3}}{3 \theta^{\frac{1}{2}}}, \quad
B_{1} = A_{4} + A_{5} A_{6}, \quad B_{2} = A_{5} A_{7}, \quad B_{3} = \frac{3}{2} A_{5}, \quad B_{4} = A_{5}A_{8}.
\end{split}
\end{equation}
%

Using the inequality $\sqrt{(x+ y + \cdots)} \leq \sqrt{x} + \sqrt{y} + \cdots$ we have
\begin{equation*}
\begin{split}
\sum_{j= 1}^{J} \frac{1}{X_{j-1}^{\frac{1}{2}}} |S_{j}| \leq \frac{k^{\frac{1}{2}}}{A_{0}^{\frac{1}{2}}}\bigg( (&\sqrt{B_{1}} t^{\frac{1}{6}} + \sqrt{B_{3}}) \sum_{j=1}^{J} 1 \\
&+ \sqrt{B_{2}}t^{\frac{1}{6}} \sum_{j=1}^{J} k^{-\frac{1}{2} j} + \sqrt{B_{4}} \sum_{j=1}^{J} k^{\frac{1}{2} j}\bigg).
\end{split}
\end{equation*}
Since
\begin{equation*}
\sum_{j=1}^{J} k^{-\frac{1}{2} j} = k^{-\frac{1}{2}} \left( \frac{1 - k^{-\frac{1}{2}J}}{1 - k^{-\frac{1}{2}}}\right), 
\end{equation*}
this gives
\begin{equation*}
\sum_{j=1}^{J} \frac{1}{X_{j-1}^{\frac{1}{2}}} |S_{j}| \leq \left( \frac{k}{A_{0}}\right)^{\frac{1}{2}}(C_{1} t^{\frac{1}{6}} \log t + C_{2} t^{\frac{1}{6}} + C_{3} t^{\frac{1}{12}} + C_{4} \log t + C_{5}),
\end{equation*}
where
\begin{equation*}
\begin{split}
C_{1} &= \frac{\sqrt{B_{1}}}{6 \log k}, \quad C_{2} = \sqrt{B_{1}}\left( 1  -  \frac{\log \left(A_{0} (2\pi)^{\frac{1}{2}}\right)}{\log k}\right) + \frac{\sqrt{B_{2}} k^{-\frac{1}{2}}}{1 - k^{-\frac{1}{2}}}\\
C_{3} &=\frac{\sqrt{B_{4}} k}{A_{0}^{\frac{1}{2}} (2\pi)^{\frac{1}{4}}(k^{\frac{1}{2}} -1)} - \frac{\sqrt{B_{2}} A_{0}^{\frac{1}{2}} (2\pi)^{\frac{1}{4}}}{k(1-k^{-\frac{1}{2}})}, \quad C_{4} = \frac{\sqrt{B_{3}}}{6 \log k}\\
C_{5} &= \sqrt{B_{3}}\left( 1 -  \frac{\left(\log A_{0} (2\pi)^{\frac{1}{2}}\right)}{\log k}\right) - \frac{\sqrt{B_{4}} k^{\frac{1}{2}}}{k^{\frac{1}{2}} - 1}.
\end{split}
\end{equation*}
This means that
\begin{equation*}
|\zeta(\tfrac{1}{2} + it)| \leq D_{1} t^{\frac{1}{6}} \log t + D_{2} t^{\frac{1}{6}} + D_{3} t^{\frac{1}{12}} + D_{4} \log t +D_{5},
\end{equation*}
where
\begin{equation}\label{morse}
\begin{split}
D_{1} &= 2 C_{1}\left( \frac{k}{A_{0}}\right)^{\frac{1}{2}}, \quad D_{2} = 2\left( 2A_{0}^{\frac{1}{2}} + C_{2}\left( \frac{k}{A_{0}}\right)^{\frac{1}{2}} \right), \quad
D_{3} = 2C_{3}\left( \frac{k}{A_{0}}\right)^{\frac{1}{2}},\\
D_{4} &= 2C_{4}\left( \frac{k}{A_{0}}\right)^{\frac{1}{2}},\quad
D_{5} = 2 \left( C_{5}\left( \frac{k}{A_{0}}\right)^{\frac{1}{2}} - 1 + \frac{0.77}{t_{0}^{\frac{1}{4}}} + \frac{1.62}{t_{0}^{\frac{3}{4}}}\right).
\end{split}
\end{equation}
To reduce the right side of (\ref{morse}) as much as possible it is desirable to choose a large value of $t_{0}$. We shall, in the next section, use (\ref{lehm}) to handle smaller values of $t$. With this in mind, the choice
\begin{equation*}
k= 1.16, \quad \theta = 7.5, \quad A_{0} = 3.37, \quad t_{0} = 5.867\times 10^{9}
\end{equation*}
means that $|\zeta(\frac{1}{2} + it)|\leq 0.732 t^{\frac{1}{6}}\log t$ for $t\geq t_{0}$, (\ref{con1}) is satisfied, and that $M\geq 2$. We now turn our attention to $t< 5.867\times 10^{9}$.


\section{Proof of Theorem \ref{t1} for small $t$}\label{sec small}

\begin{Lem}\label{lem:comp1}
For $t\in[2,5.867\times 10^9]$ we have
\begin{equation*}
\left|\zeta\left(\tfrac{1}{2}+it\right)\right|<0.732 t^{\frac{1}{6}}\log t.
\end{equation*}
\begin{proof}
The trivial bound (\ref{lehm}) is tighter than our new bound at $t=5.867\times 10^9$  and remains so for $t$ all the way down to $t=226.7088\ldots$. We checked the range $[2,230]$ rigorously by computer as follows.

We implemented an interval arithmetic version of the Euler--MacLaurin summation formula that, given an interval $\underline{t}$ returns an interval that includes $|\zeta(\frac{1}{2}+it)|$ for all $t\in\underline{t}$. We divided the line segment $[2,230]$ into pieces of length $1/1024$ and for each piece, checked that $|\zeta(\frac{1}{2}+it)|$ did not exceed our bound. Specifically, if we are considering $\underline{t}=[a,a+1/1024]$ and we know that for $t\in\underline{t}$ that $|\zeta(\frac{1}{2}+it)|\in[x,y]$, then we check $y<0.732 a^{\frac{1}{6}} \log a$. No counter examples exist for $t\in[2,230]$ and this establishes the lemma.
\end{proof}
\end{Lem}

\begin{cor}
For $t$ real and $Q\geq 4.678$ we have 
\begin{equation*}
\left|\zeta\left(\tfrac{1}{2}+it\right)\right|<0.732 |Q+it|^{\frac{1}{6}}\log |Q+it|.
\end{equation*}
\begin{proof}
For $|t|\geq 2$ we use Lemma \ref{lem:comp1}. For $t\in(-2,2)$ we know that $|\zeta(\frac{1}{2}+it)|$ attains a maximum at $t=0$ so we determine a $Q$ such that
\begin{equation*}
|\zeta(\tfrac{1}{2})|<0.732 Q^{\frac{1}{6}}\log Q
\end{equation*}
and we are done.
\end{proof}
\end{cor}

\section{Conclusion}\label{secconc}
Since an Euler--MacLaurin computation of $\zeta(\frac{1}{2}+it)$ becomes inefficient as $t$ increases, we also implemented an interval version of the Riemann--Siegel formula (R-S) for $t\geq 200$. Above this height we have explicit error bounds due to Gabcke \cite{Gabcke}. The only nuance is that the main sum of R-S runs from $1$ to $\lfloor \sqrt{t/2\pi} \rfloor$ and we must be careful not to compute with intervals $\underline{t}=[a,b]$ such that $\lfloor \sqrt{a/2\pi}\rfloor \neq \lfloor\sqrt{b/2\pi}\rfloor$. We get around this by using Euler--MacLaurin for such intervals.

So armed, we can continue to compute $|\zeta(\frac{1}{2}+it)|$ for $t\in[a,b]$ and each time we come across an interval where (possibly) $|\zeta(\frac{1}{2}+i[a,b])|$ sets a new record $[x,y]$, we store $a$ and $y$. Running through the data files produced, it is a trivial matter to find an $A$ such that $|\zeta(\frac{1}{2}+it)|<A t^{\frac{1}{6}}\log t$ throughout the range. Our results are summarised in Table \ref{tab:abounds}.

\begin{table}[ht]
\caption{Bounds on $|\zeta(\frac{1}{2} +it)| \leq A t^{\frac{1}{6}}\log t$ for ranges of $t$.}
\centering
\begin{tabular}
{c c c c c c}
\hline\hline
$\underline{t}$& $A$\\[0.5ex]\hline
$[2,200]$&$0.7090$\\
$[200,10^3]$ & $0.4873$\\
$[10^3,10^4]$ & $0.4682$\\
$[10^4,10^5]$ &  $0.4217$\\
$[10^5,10^6]$ & $0.3765$\\
$[10^6,10^7]$ & $0.3238$\\
$[10^7,10^8]$ & $0.2854$\\
    \hline\hline
  \end{tabular}
\label{tab:abounds}
\end{table}
It seems that the bound in Theorem \ref{t1} is still very far from optimal.

\bibliographystyle{plain}
\bibliography{themastercanada}

\end{document}